\documentclass[10pt]{amsart}
\usepackage{amsmath,amscd}
\usepackage{amsbsy}
\usepackage{amssymb}
\usepackage{amscd,amsthm}
\usepackage[all,cmtip]{xy}
\usepackage[russian,english]{babel}
\usepackage[utf8]{inputenc}

\newtheorem{thm}{Theorem}
\numberwithin{thm}{section}
\newtheorem{lem}[thm]{Lemma}
\newtheorem{prop}[thm]{Proposition}
\newtheorem{cor}[thm]{Corollary}
\newtheorem{exam}[thm]{Example}
\newtheorem{rema}[thm]{Remark}

\newtheorem{defi}[thm]{Definition}

\newtheorem*{thm2}{Theorem}
\newtheorem*{cor2}{Corollary}

\begin{document}
\begin{center}
\huge{Amitsur subgroup and noncommutative motives}\\[1cm]
\end{center}

\begin{center}

\large{Sa$\mathrm{\check{s}}$a Novakovi$\mathrm{\acute{c}}$}\\[0,4cm]
{\small January 2021}\\[0,3cm]
\end{center}
\begin{center}
\begin{otherlanguage*}{russian}
	\emph{ЗА ИЛИАНУ.}	
\end{otherlanguage*}
\end{center}

\noindent{\small \textbf{Abstract}. 
This paper addresses the problem of calculating the Amitsur subgroup of a proper $k$-scheme. Under mild hypothesis, we calculate this subgroup for proper $k$-varieties $X$ with $\mathrm{Pic}(X)\simeq \mathbb{Z}^{\oplus m}$, using a classification of so called absolutely split vector bundles ($AS$-bundles for short). We also show that the Brauer group of $X$ is isomorphic to $\mathrm{Br}(k)$ modulo the Amitsur subgroup, provided $X$ is geometrically rational. Our results also enable us to classify $AS$-bundles on twisted flags. Moreover, we find an alternative proof for a result due to Merkurjev and Tignol, stating that the Amitsur subgroup of twisted flags is generated by a certain subset of the set of classes of Tits algebras of the corresponding algebraic group. This result of Merkurjev and Tignol is actually a corollary of a more general theorem that we prove. The obtained results have also consequences for the noncommutative motives of the twisted flags under consideration. In particular, we show that a certain noncommutative motive of a twisted flag is a birational invariant, generalizing in this way a result of Tabuada. We generalize this result for $X$ having a certain type of semiorthogonal decomposition.}\\

\section{Introduction}
Let $f\colon X\rightarrow S$ be a scheme that is seperated and of finite type over a Noetherian scheme $S$ and assume $\mathcal{O}_S\xrightarrow{\simeq } f_*\mathcal{O}_X$. Then, for each $S$-scheme $T$ there is an exact sequence
\begin{eqnarray*}
0\longrightarrow \mathrm{Pic}(T)\longrightarrow \mathrm{Pic}(X_T)\longrightarrow \mathrm{Pic}_{(X/S)(\mathrm{fppf})}(T)\stackrel{\delta}{\longrightarrow}\mathrm{Br}'(T)\longrightarrow \mathrm{Br}'(X_T).
\end{eqnarray*}	
Here $\mathrm{Pic}_{(X/S)}$ denotes the Picard functor and $\mathrm{Pic}_{(X/S)(\mathrm{fppf})}$ the associated sheaf in the fppf topology. Specializing the above sequence to the case $X$ a proper variety over a field $k$ and $T=\mathrm{Spec}(k)$, Liedtke \cite{LIT} called the group $\delta(\mathrm{Pic}_{(X/S)(\mathrm{fppf})}(k))$ the \emph{Amitsur subgroup} of $X$ in $\mathrm{Br}(k)$. This subgroup is denoted by $\mathrm{Am}(X)$. If $X$ is smooth and proper over $k$, it follows that $\mathrm{Am}(X)=\mathrm{ker}(\mathrm{Br}(k)\rightarrow \mathrm{Br}(k(X)))$. The group $\mathrm{ker}(\mathrm{Br}(k)\rightarrow \mathrm{Br}(k(X)))$ is also denoted by $\mathrm{Br}(k(X)/k)$ and was studied for instance in \cite{METI} and \cite{CT}. If $X$ is Brauer--Severi corresponding to a central simple algebra $A$, it is a classical result due to Ch\^{a}telet that $\mathrm{Am}(X)=\langle [A]\rangle$. In this case $\mathrm{Br}(X)\simeq \mathrm{Br}(k)/\mathrm{Am}(X)$. It is shown in \cite{LIT} that $\mathrm{Am}(X)$ is a birational invariant for $X$ smooth and proper over $k$. Note that if $X$ admits a $k$-rational point, then $\mathrm{Am}(X)=0$. On the other hand, there are proper varieties with trivial Amitsur subgroup without rational points (see \cite{LIT}, Proposition 5.4). Special varieties for which $\mathrm{Am}(X)$ is calculated can be found in \cite{CT}, \cite{LIT} and \cite{METI}. In this paper, we want to calculate $\mathrm{Am}(X)$ and $\mathrm{Br}(X)$ for a certain class of proper $k$ schemes $X$. Furthermore, we want to explain the consequences of our results for the noncommutative motives.

To state our results, we fix some notations and recall some facts.
Let $X$ be a proper and geometrically integral $k$-scheme and denote by $X_s$ the base change $X\otimes_k k^s$ to the separable closure. If $\mathrm{Pic}(X_s)\simeq \mathbb{Z}^{\oplus m}$, then $\mathrm{Pic}(X)\simeq r_1\mathbb{Z}\oplus \cdots \oplus r_m\mathbb{Z}$.  Let us fix a basis $\mathcal{L}_1,...,\mathcal{L}_m$ of $\mathrm{Pic}(X_s)$. 
 Using a basis of $\mathrm{Pic}(X)$ one can show that there are line bundles $\mathcal{J}_i\in\mathrm{Pic}(X)$ 
satisfying $\mathcal{J}_i\otimes_k k^s\simeq \mathcal{L}_i^{\otimes c_i}$ for some integers $c_i\geq 1$. Now choose the line bundles $\mathcal{J}_i$ such that the $c_i$ are minimal. According to \cite{NO}, Proposition 3.4, these $\mathcal{J}_i$ are unique up to isomorphism. 
Assume there are pure vector bundles $\mathcal{M}_i$ of type $\mathcal{L}_i\in \mathrm{Pic}(X_s)$. A vector bundle $\mathcal{E}$ on a proper $k$-scheme $X$ is called \emph{pure of type} $\mathcal{W}$ if there is an indecomposable vector bundle $\mathcal{W}$ on $X\otimes_k \bar{k}$ such that $\mathcal{E}\otimes_k {\bar{k}}\simeq \mathcal{W}^{\oplus m}$. Being pure of type $\mathcal{L}_i$ is equivalent to $\mathcal{L}_i\in \mathrm{Pic}_{\Gamma}(X_s)$, where $\Gamma$ denotes the absolute Galois group (see \cite{NO}, Theorem 4.5). We know from \cite{NO}, Proposition 3.5 that the bundle $\mathcal{M}_i$ is unique up to isomorphism. We set $\mathcal{M}_{\mathcal{L}_i}:=\mathcal{M}_i$. It is easy to see that for any line bundle $\mathcal{L}_i^{\otimes a}\in \mathrm{Pic}(X_s)$ there is an indecomposable pure bundle of type $\mathcal{L}_i^{\otimes a}$. Indeed, let $s_i=\mathrm{rk}(\mathcal{M}_{\mathcal{L}_i})$ and consider $(\mathcal{L}_i^{\oplus s_i})^{\otimes a}\simeq (\mathcal{L}_i^{\otimes a})^{\oplus s_i^a}$. Then we get 
$\mathcal{M}_{\mathcal{L}_i}^{\otimes a}\otimes_k k^s\simeq (\mathcal{L}_i^{\oplus s_i})^{\otimes a}\simeq (\mathcal{L}_i^{\otimes a})^{\oplus s_i^a}$.

\noindent
Considering the Krull--Schmidt decomposition of $\mathcal{M}_{\mathcal{L}_i}^{\otimes a}$ and taking into account that all indecomposable direct summands are isomorphic (see \cite{NO}, proof of Proposition 3.6 and Remark 3.7), we get an, up to isomorphism, unique indecomposable vector bundle $\mathcal{M}_{\mathcal{L}_i^{\otimes a}}$ such that $\mathcal{M}_{\mathcal{L}_i^{\otimes a}}\otimes_k k^s\simeq (\mathcal{L}_i^{\otimes a})^{\oplus s_i(a)}$, where $s_i(a)=\mathrm{rank}(\mathcal{M}_{\mathcal{L}_i^{\otimes a}})$. 
Using Krull--Schmidt decomposition again, we can use our indecomposable vector bundles $\mathcal{M}_{\mathcal{L}_i^{\otimes a}}$ and take the tensor product of these to get an indecomposable vector bundle $\mathcal{M}_{(a_1,...,a_m)}$ of type $\mathcal{L}_1^{\otimes a_1}\otimes\cdots \otimes \mathcal{L}_m^{\otimes a_m}$. Again, the bundles $\mathcal{M}_{(a_1,...,a_m)}$ are unique up to isomorphism. Recall that a vector bundle $\mathcal{E}$ on a $k$-scheme $X$ is called \emph{absolutely split} if it splits after base change as a direct sum of line bundles on $X\otimes_k \bar{k}$. For an absolutely split vector bundle we shortly write \emph{AS-bundle}. Over an algebraically closed field, a classical result of Grothendieck classifies all $AS$-bundles on $\mathbb{P}^1$. Note that on $\mathbb{P}^1$ the result of Grothendieck shows that actually all vector bundles are $AS$-bundles. In \cite{NOO} the author classifies vector bundles on twisted forms of $\mathbb{P}^1$. The twisted forms of $\mathbb{P}^1$ are Brauer--Severi curves (or smooth non-degenerate quadrics without rational point). Generalizing these results further, in \cite{NO}, the author clssifies $AS$-bundles on proper $k$-schemes. Moreover, for $X$ with cyclic Picard group the indecomposable $AS$-bundles are determined explicitely. In the present paper we generalize this result to the case $\mathrm{Pic}(X_s)\simeq \mathbb{Z}^{\oplus m}$.

\begin{thm2}[Theorem 4.5] 
	Let $X$ be a proper and geometrically integral $k$-scheme with $\mathrm{Pic}(X_s)\simeq \mathbb{Z}^{\oplus m}$ and let $\mathcal{L}_1,...,\mathcal{L}_m\in \mathrm{Pic}(X_s)$ be the basis from above. Let $\mathcal{J}_i\in \mathrm{Pic}(X)\simeq \mathbb{Z}^{\oplus m}$ be the up to isomorphism unique line bundles satisfying $\mathcal{J}_i\otimes_k k^s\simeq \mathcal{L}_i^{\otimes c_i}$ with $c_i$ being minimal. Assume there are indecomposable pure bundles $\mathcal{M}_{\mathcal{L}_i}$ of type $\mathcal{L}_i$. Then all indecomposable $AS$-bundles $\mathcal{E}$ are of the form
	\begin{eqnarray*}
		\mathcal{J}_1^{\otimes b_1}\otimes\cdots \otimes \mathcal{J}_m^{\otimes b_m} \otimes \mathcal{M}_{(a_1,...,a_m)}
	\end{eqnarray*}
	with unique $b_i\in\mathbb{Z}$ and $0\leq a_j\leq c_j-1$.
\end{thm2} 
Notice that Theorem 4.5 is a generalization of \cite{NO}, Theorem 5.1. Using Theorem 4.5 we will prove the following result about the Amitsur subgroup and the Brauer group $\mathrm{Br}(X)$.
\begin{thm}
Let $X$ be proper and geometrically integral $k$-scheme with $\mathrm{Pic}(X_s)\simeq \mathbb{Z}^{\oplus m}$. With the notation and assumption from above, the Amitsur subgroup $\mathrm{Am}(X)$ is generated by the classes of central simple algebras $\mathrm{End}(\mathcal{M}_{e_i})$, where $i\in\{1,...,m\}$ for which $c_i\geq 2$. In particular, if $\mathrm{Pic}(X_s)\simeq \mathbb{Z}$ is generated by an ample line bundle $\mathcal{L}$ and if there exists a pure vector bundle $\mathcal{M}_{\mathcal{L}}$ of type $\mathcal{L}$, then $\mathrm{Am}(X)$ is cyclic and generated by $[\mathrm{End}(\mathcal{M}_{\mathcal{L}})]$. Moreover, if $X$ is geometrically rational, then $\mathrm{Br}(X)\simeq \mathrm{Br}(k)/\mathrm{Am}(X)$.
\end{thm}
Schemes satisfying the assuptions of Theorem 1.1 include twisted flag varieties. In order to apply Theorem 1.1 to twisted flags and to obtain in this way an alternative proof of a result of Merkurjev and Tignol \cite{METI}, Theorem B, we recall the definition and some facts on twisted flags and refer to \cite{MPW} for details. 

Let $G$ be a semisimple algebraic group over a field $k$ and $G_s=G\otimes_k k^{s}$. For a parabolic subgroup $P$ of $G_s$, one has a homogeneous variety $G_s/P$. A \emph{twisted flag} is variety $X$ such that $X\otimes_k k^{s}$ is $G_s$-isomorphic to $G_s/P$ for some $G$ and some parabolic $P$ in $G_s$. Any twisted flag is smooth, absolutely irreducible and reduced. An algebraic group $G'$ is called twisted form of $G$ iff $G'_s\simeq G_s$ iff $G'={_\gamma} G$ for some $\gamma\in Z^1(k,\mathrm{Aut}(G_s))$. The group $G'$ is called an \emph{inner form} of $G$ if there is a $\delta\in Z^1(k,\bar{G}(k^{s}))$ with $G'={_\delta}G$. Here $\bar{G}=G/Z(G)$ where $Z(G)$ denotes the center. For an arbitrary semisimple $G$ over $k$, there is a unique (up to isomorphism) split semisimple group $G^d$ such that $G_s\simeq G_s^d$. If $G$ is an inner form of $G^d$, then $G$ is said to be of \emph{inner type}. For instance, let $A$ be a central simple algebra over $k$ of degree $n$ and $G=\mathrm{PGL}_1(A)$, then $G_s\simeq \mathrm{PGL}_n$ over $k^{s}$. Hence $G$ is an inner form of $\mathrm{PGL}_n$. Since $\mathrm{PGL}_n$ is split, $G=\mathrm{PGL}_1(A)$ is of inner type. For a classification of simple groups of classical type we refer to \cite{KNU}, p.366-373.

\noindent
Now let $G$ be a semisimple (so connected) simply connected algebraic group over $k$ and $P$ a parabolic subgroup. Let $G/P$ be a flag variety and note that $G/P=\bar{G}/\bar{P}$. Let $\gamma\colon \mathrm{Gal}(k^s|k)\rightarrow G(k^s)$ be a 1-cocycle. We denote by $X:={_\gamma}(G/P)$ the twisted form of $G/P$ corresponding to $\gamma$. Notice that ${_\gamma}(G/P)\otimes_k k^s\simeq G_s/P_s$ for a suitable parabolic subgroup $P_s$ of $G_s$. The next corollary is essentially \cite{METI}, Theorem B. Below $\mathrm{Ch}(P_s)$ denotes the character group and $\mathrm{Ch}(P_s)^{\Gamma}$ the character group of Galois invariant characters.
\begin{cor}
	Let $G$ be a semisimple (so connected) simply connected algebraic group over $k$ and $P$ a parabolic subgroup. Denote by $X={_\gamma}(G/P)$ a twisted flag. Then $\mathrm{Am}(X)$ is generated by the Brauer classes of Tits algebras of $G$ corresponding to the elements of a basis of $\mathrm{Ch}(P_s)^{\Gamma}$. 
	Moreover, $\mathrm{Br}(X)\simeq \mathrm{Br}(k)/\mathrm{Am}(X)$.
\end{cor}

\noindent
The proof of Theorem 1.1 actually uses Theorem 4.5 which enables us to classify all indecomposable $AS$-bundles on $X$. It is a non-trivial fact that these bundles are in one-to-one correspondence with the closed points of the Picard scheme $\mathrm{Pic}_{(X/k)(\mathrm{fppf})}$ (see Theorem 4.4). Notice that Theorem 4.5 generalizes \cite{NO}, Theorem 5.1.
\begin{cor}
	Let $X_i$ be a twisted form of $G_i/P_i$ with $G_i$ and $P_i$ as in Theorem 1.1 and let $X=X_1\times\cdots \times X_n$. Let $D_i$ be the set of generators of $\mathrm{Am}(X_i)$ obtained from Corollary 1.2. Then $\mathrm{Am}(X)$ is generated by $\cup D_i$. Moreover, $\mathrm{Br}(X)\simeq \mathrm{Br}(k)/\mathrm{Am}(X)$.
\end{cor}
\noindent
Theorem 1.1 from above has also a motivic consequence. Noncommutative motives are by construction closely related to semiorthogonal decompositions. In the last decades, the bounded derived category $D^b(X)$ of coherent sheaves on a smooth projective variety $X$ has been recognized as an interesting invariant, encoding a lot of geometric information. For instance, there are links between the semiorthogonal decomposition of $D^b(X)$ and the birational geometry of $X$ (see for instance \cite{KUS}, \cite{AB1S}, \cite{ABS}, \cite{NO2Z}, \cite{NO3T} and references therein). From a motivic point of view, it is quite natural to ask how birational geometry of a given variety $X$ is detected by its noncommutative motive. And indeed, there are results in this direction for (generalized) Brauer--Severi varieties \cite{TAZ} and \cite{TA2S}. In the present paper we want to consider twisted flags and shed some light to the case of arbitrary proper $k$-schemes admitting a certain type of semiorthogonal decomposition. Our main results are Theorems 1.4 and 1.7.

Recall from the book \cite{GTAZ} that the category $\textbf{dgcat}$ of small dg categories with dg functors carries a Quillen model structure whose weak equivalences are Morita equivalences. Denote by $\mathrm{Hmo}$ the homotopy category obtained from the Quillen model structure and by $\mathrm{Hmo}_0$ its additivization. To any small dg category $\mathcal{A}$ one can associate functorially its \emph{noncommutative motive} $U(\mathcal{A})$ which takes values in $\mathrm{Hmo}_0$. This functor $U\colon \textbf{dgcat}\rightarrow \mathrm{Hmo}_0$ is an \emph{universal additive invariant}. Recall that an universal additive invariant is any functor $E\colon \textbf{dgcat}\rightarrow D$ taking values in an additive category $D$ such that\\ 
\begin{itemize}
	\item[(\textbf{i})] it sends derived Morita equivalences to isomorphisms,\\
	
	\item[(\textbf{ii})] for any pre-triangulated dg category $\mathcal{A}$ admitting full pre-triangulated dg subcategories $\mathcal{B}$ and $\mathcal{C}$ such that $H^0(\mathcal{A})=\langle H^0(\mathcal{B}), H^0(\mathcal{C})\rangle$ is a semiorthogonal decomposition, the morphism $E(\mathcal{B})\oplus E(\mathcal{C})\rightarrow E(\mathcal{A})$ induced by the inclusions is an isomorphism.\\
\end{itemize}
\noindent
A source of examples for dg categories is provided by schemes since the derived category of perfect complexes $\mathrm{perf}(X)$ of any quasi-projective scheme $X$ admits a canonical (unique) dg enhancement $\mathrm{perf}_{dg}(X)$. In \cite{TAZ} it is proved that if two Brauer--Severi varieties $X$ and $Y$ (see Section 2 for a definition) are birational, then $U(\mathrm{perf}_{dg}(X))=U(\mathrm{perf}_{dg}(Y))$. In view of the Amitsur conjecture for central simple algebras (two Brauer--Severi varieties $X$ and $Y$ are birational if and only if the corresponding central simple algebras $A$ and $B$ generate the same subgroup in $\mathrm{Br}(k)$), it is conjectured in \emph{loc.cite} that $U$ is actually a complete birational invariant for Brauer--Severi varieties. As a Brauer--Severi variety is a special case of a twisted flag, Theorem 1.4 from below is a generalization. 

We fix some notation:  Let $X$ be a twisted flag as in Corollary 1.2 and denote by $A_g$ the central simple division algebra corresponding to $g\in\mathrm{Am}(X)$. Analogously, let $B_h$ denote the central simple division algebra corresponding to $h\in\mathrm{Am}(Y)$. Set $M_X:=\bigoplus _{g\in \mathrm{Am}(X)}U(A_{g})$ and $M_Y:=\bigoplus_{h\in\mathrm{Am}(Y)}U(B_{h})$. Note that these sums are finite according to Theorem 1.1. Furthermore, let $\mathrm{Sep}(k)$ be the full subcategory of the category of noncommutative Chow motives (see \cite{TAZ} for a definition) consisting of objects $U(F)$ with $F$ a separable $k$-algebra. Now let $\mathrm{CSA}(k)$ be the full subcategory of $\mathrm{Sep}(k)$ consisting of objects $U(A)$ with $A$ a central simple $k$-algebra (see Section 2 for a definition of central simple algebra) and denote by $\mathrm{CSA}(k)^{\oplus}$ its closure under finite sums. It is an additive symmetric monoidal subcategory. We write shortly $U(X)$ for $U(\mathrm{perf}_{dg}(X))$.

Let $G$, $P$ and $\gamma$ be as above and let $\rho_1,...,\rho_n$ be a $\mathrm{Ch}$-homogeneous basis of $R(P)$ over $R(G)$ (see \cite{PAS},\S2), where $R(P)$ and $R(G)$ denote the corresponding representation rings. Let $A_{\chi(i),\gamma}$ be the Tits central simple algebras associated to $\rho_i$ (see Section 3 for a definition) and let $\mathrm{Ti}(X):=\langle A_{\chi(1),\gamma},...,A_{\chi(n),\gamma} \rangle$ be the subgroup of $\mathrm{Br}(k)$ generated by these Tits algebras. Denote by 
\begin{eqnarray*}
	M\mathrm{Ti}(X):=\bigoplus_{f\in \mathrm{Ti}(X)}U(A_f),
\end{eqnarray*}
where $A_f$ are the central simple division algebras corresponding to $f$.

\begin{thm}
	Let $X$ be a twisted flags as in Corollary 1.2. Then there are direct summands $N,N'\in \mathrm{CSA}(k)^{\oplus}$ of $	M\mathrm{Ti}(X)$ such that $M_X\oplus N=U(X)\oplus N'$.
\end{thm}
\begin{rema}
	\textnormal{In the case of a Brauer--Severi variety $X$ corresponding to a central simple algebra of period $m$, one has $M_X=U(k)\oplus U(A)\oplus\cdots \oplus U(A^{\otimes m-1})$. Note that the period $m$ divides the degree $n$. So if $m\cdot r=n$, we get $M_X^{\oplus r}=U(\mathrm{perf}(X))$ (see Example 6.1 for details). In this case we have $N=M_X^{\oplus (r-1)}$ and $N'=0$. With the help of Corollary 1.6 we get back \cite{TAZ}, Proposition 3.15 (see p.14 for detailed explanation).}
\end{rema}
\begin{cor}
	Let $X$ and $Y$ be twisted flags as in Corollary 1.2 and let $N,N'$ be the direct summands of $M\mathrm{Ti}(X)$ and $Q,Q'$ of $M\mathrm{Ti}(Y)$ obtained from Theorem 1.4. If $X$ and $Y$ are birational, then $M_X\simeq M_Y$ and $U(X)\oplus N'\oplus Q\simeq U(Y)\oplus Q'\oplus N$. 
\end{cor}
Using the theory of semiorthogonal decompositions one can try to generalize Corollary 1.6 for arbitrary proper and geometrically integral $k$-schemes. For the definition of w-exceptional objects and semiorthogonal decompositions we refer to p.15. Below, $D^b(X)$ denotes the bounded derived category of coherent sheaves on $X$ and $D^b(X)=\langle \mathcal{D}_1,...,\mathcal{D}_m\rangle$ a semiorthogonal decomposition. We want to call a smooth, proper and geomerically integral $k$-scheme $X$ a scheme of \emph{pure weak exceptional type} if it satisfies the following conditions:\\
\begin{itemize}
	\item[(\textbf{i})] $D^b(X)=\langle \mathcal{E}_1,...,\mathcal{E}_m\rangle$ is a semiorthogonal decomposition induced by a full w-exceptional collection,\\
	
	\item[(\textbf{ii})]  The bundles $\mathcal{E}_i$ be pure of type $\mathcal{K}_i$. Furthermore, $\mathrm{End}(\mathcal{K}_i)\simeq k^s$ and some of these $\mathcal{K}_i$ form a basis of $\mathrm{Pic}(X_s)$.\\
\end{itemize}
Note that if $\mathcal{E}_i$ is pure of type $\mathcal{K}_i$ with $\mathrm{End}(\mathcal{K}_i)\simeq k^s$, the base change of the semiorthogonal decomposition (i) actually implies that $D^b(X_s)=\langle\mathcal{K}_1,...,\mathcal{K}_m\rangle$ is a full exceptional collection (see p.15 for a definition) and hence $\mathrm{Pic}(X_s)\simeq \mathbb{Z}^{\oplus m}$. Indeed, there are schemes satisfying these two conditions. For instance (generalized) Brauer-Severi varieties or certain involution varieties. Conjecturally, all twisted flags are of pure weak exceptional type. It is interesting to investigate whether twisted flags are characterized by (i) and (ii). For a result in this direction see \cite{NO5}, Theorem 1.2. One can show that schemes $X$ of pure weak exceptional type satisfy the assumptions of Theorem 4.5 (see proof of Theorem 1.7). Therefore, we can classify indecomposable $AS$-bundles on such $X$.

Now let $D_X$ denote the set of indecopmposable $AS$-bundles on a scheme of pure weak exceptional type. Note that for any $AS$-bundle $\mathcal{E}$ on $X$ satisfying the assumptions of Theorem 4.5 one has $[\mathrm{End}(\mathcal{E})]\in \mathrm{Br}(k)$ (see Proposition 4.7). From Theorem 4.5 and Proposition 4.7 we conclude that there are only finitely many Brauer-classes $[\mathrm{End}(\mathcal{E})]\in \mathrm{Br}(k)$ of indecomposable $AS$-bundles. 
Now let $C_X\subset \mathrm{Br}(k)$ be the subgroup generated by these finitely many Brauer-classes. Denote by $A(g)$ the central simple division algebra in $C_X$ corresponding to $g$. Set 
\begin{eqnarray*}
	M_{T(X)}:=\bigoplus _{g\in C_X}U(A(g))
\end{eqnarray*}
\begin{thm}
Let $X$ and $Y$ be schemes of pure weak exceptional type. 
If $X$ and $Y$ are birational, then there are direct summands $N,N'\in \mathrm{CSA}^{\oplus }$ of $M_{T(X)}$ and $Q,Q'\in \mathrm{CSA}^{\oplus }$ of $M_{T(Y)}$ such that $U(X)\oplus N'\oplus Q\simeq U(Y)\oplus Q'\oplus N$.
\end{thm}
\begin{cor}
	Let $X$ and $Y$ be as in Theorem 1.7. Assume $X$ and $Y$ are birational and let $A_i, 1\leq i\leq n$ and $B_j, 1\leq j\leq m$ be the central simple algebras occuring in $U(X)\oplus N'\oplus Q$ and $U(Y)\oplus Q'\oplus N$ respectively, then $\langle [A_i]\rangle=\langle[B_j]\rangle$ in $\mathrm{Br}(k)$.\\
\end{cor}

\noindent{\small \textbf{Notations} Throughout the paper $k$ is an arbitrary field and $k^s$ a separable closure. For a variety/algebraic group over $k$, we write $X_s$ and $G_s$ for the base changes $X\otimes_k k^s$ and $G\otimes_k k^s$ respectively.}



\section{Examples of inner twisted flags}
Recall that a finite-dimensional $k$-algebra $A$ is called \emph{central simple} if it is an associative $k$-algebra that has no two-sided ideals other than $0$ and $A$ and if its center equals $k$. If the algebra $A$ is a division algebra it is called \emph{central division algebra}. Note that $A$ is a central simple $k$-algebra if and only if there is a finite field extension $k\subset L$, such that $A\otimes_k L \simeq M_n(L)$. This is also equivalent to $A\otimes_k \bar{k}\simeq M_n(\bar{k})$. An extension $k\subset L$ such that $A\otimes_k L\simeq M_n(L)$ is called splitting field for $A$. The \emph{degree} of a central simple algebra $A$ is defined to be $\mathrm{deg}(A):=\sqrt{\mathrm{dim}_k A}$. According to the \emph{Wedderburn Theorem}, for any central simple $k$-algebra $A$ there is an unique integer $n>0$ and a division $k$-algebra $D$ such that $A\simeq M_n(D)$. The division algebra $D$ is also central and unique up to isomorphism. The degree of the unique central division algebra $D$ is called the \emph{index} of $A$ and is denoted by $\mathrm{ind}(A)$. Two central simple algebras $A$ and $B$ are said to be \emph{Brauer-equivalent} if there are positive integers $r,s$ such that $M_r(A)\simeq M_s(B)$.

For a central simple $k$-algebra $A$, the inner twisted forms arising from $G=\mathrm{PGL}_1(A)$ can be described very explicitly. This will be done in the sequel. One of these inner twisted forms is the \emph{generalized Brauer--Severi variety}. 
So let $m\leq n$. The generalized Brauer--Severi variety $\mathrm{BS}(m,A)$ is defined to be the subset of $\mathrm{Grass}_k(mn,A)$ consisting of those subspaces of $A$ which are right ideals of dimension $m\cdot n$ (see \cite{KNU} or \cite{BLS}). Recall that $\mathrm{Grass}_k(mn,A)$ is given the structure of a projective variety via the Pl\"ucker embedding $\mathrm{Grass}_k(mn,A)\rightarrow \mathbb{P}(\wedge^{mn}(A))$.
This gives an embedding $\mathrm{BS}(m,A)\rightarrow \mathbb{P}(\wedge^{mn}(A))$ and a very ample line bundle $\mathcal{M}$ on $\mathrm{BS}(m,A)$.
Note that for any $\mathrm{BS}(m,A)$ there exists a finite Galois field extension $E$ of $k$ such that $\mathrm{BS}(m,A)\otimes_k E\simeq \mathrm{Grass}_E(mn,n^2)\simeq \mathrm{Grass}_E(m,n)$. The Picard group $\mathrm{Pic}(\mathrm{Grass}_E(m,n))$ is isomorphic to $\mathbb{Z}$ and has ample generator $\mathcal{O}(1)\simeq \mathrm{det}(\mathcal{Q})$ with $\mathcal{Q}$ being the universal quotient bundle on $\mathrm{Grass}_E(m,n)$. Recall that $\mathrm{Pic}(\mathrm{BS}(m,A))\simeq\mathbb{Z}$ and that it has a positive generator $\mathcal{L}$ such that $\mathcal{L}\otimes_k E\simeq \mathcal{O}(r)$ for a suitable $r>0$. Since $\mathrm{Pic}(\mathrm{BS}(m,A))$ is cyclic, we have $\mathcal{L}^{\otimes s}\simeq \mathcal{M}$ for a suitable $s>0$. Therefore, $\mathcal{L}$ is ample. From the definition of $\mathrm{BS}(m,A)$ it is clear that $\mathcal{L}$ is also very ample. If $m=1$, $\mathrm{BS}(1,A)$ is called \emph{Brauer--Severi variety}.


We also recall the basics of generalized Brauer--Severi schemes (see \cite{LSW}). Let $X$ be a noetherian $k$-scheme and $\mathcal{A}$ a sheaf of Azumaya algebras of rank $n^2$ over $X$ (see \cite{GRO}, \cite{GRO1} for details on Azumaya algebras). For an integer $1\leq m_1<n$ the generalized Brauer--Severi scheme $p:\mathrm{BS}(m_1,\mathcal{A})\rightarrow X$ is defined as the scheme representing the functor $F:\mathrm{Sch}/X\rightarrow \mathrm{Sets}$, where $(\psi:Y\rightarrow X )$ is mapped to the set of left ideals $\mathcal{J}$ of $\psi^*\mathcal{A}$ such that $\psi^*\mathcal{A}/\mathcal{J}$ is locally free of rank $n(n-m_1)$. By definition, there is an \'etale covering $U\rightarrow X$ and a locally free sheaf $\mathcal{E}$ of rank $n$ with the following trivializing diagram:
\begin{displaymath}
\begin{xy}
\xymatrix{
	\mathrm{Grass}(m_1,\mathcal{E}) \ar[r]^{\pi} \ar[d]_{q}    &   \mathrm{BS}(n_1,\mathcal{A}) \ar[d]^{p}                   \\
	U \ar[r]^{g}             &   X             
}
\end{xy}
\end{displaymath}
In the same way one defines the twisted relative flag $\mathrm{BS}(m_1,...,m_r,\mathcal{A})$ as the scheme representing the functor $F:\mathrm{Sch}/X\rightarrow \mathrm{Sets}$, where $(\psi:Y\rightarrow X )$ is mapped to the set of left ideals $\mathcal{J}_1\subset...\subset \mathcal{J}_r$ of $\psi^*\mathcal{A}$ such that $\psi^*\mathcal{A}/\mathcal{J}_i$ is locally free of rank $n(n-m_i)$. As for the generalized Brauer--Severi schemes, there is an \'etale covering $U\rightarrow X$ and a locally free sheaf $\mathcal{E}$ of rank $n$ with diagram
\begin{displaymath}
\begin{xy}
\xymatrix{
	\mathrm{Flag}_U(m_1,...,m_r,\mathcal{E}) \ar[r]^{\pi} \ar[d]_{q}    &   \mathrm{BS}(m_1,...,m_r,\mathcal{A}) \ar[d]^{p}                   \\
	U \ar[r]^{g}             &   X             
}
\end{xy}
\end{displaymath}
Note that the usual Brauer--Severi schemes are obtained from the generalized one by setting $m_1=1$. In this case one has a well known one-to-one correspondence between sheaves of Azumaya algebras of rank $n^2$ on $X$ and Brauer--Severi schemes of relative dimension $n-1$ via $\check{H}^1(X_{et}, \mathrm{PGL}_n)$ (see \cite{GRO}). Note that if the base scheme $X$ is a point a sheaf of Azumaya algebras on $X$ is a central simple $k$-algebra and the generalized Brauer--Severi schemes are the generalized Brauer--Severi varieties from above. Consider a twisted flag $X=\mathrm{SB}(m_1,...,m_r, A) \rightarrow \mathrm{Spec}(k)$. Such an $X$ is an \emph{inner form} of a partial flag variety.
That is, there is a cartesian square of the form
\begin{displaymath}
\begin{xy}
\xymatrix{
	\mathrm{Grass}_L(m_1,...,m_r,V) \ar[r]^{\pi} \ar[d]_{q}    &   \mathrm{BS}(m_1,...,m_r, A) \ar[d]^{p}                   \\
	\mathrm{Spec}(L) \ar[r]^{\pi}             &   \mathrm{Spec}(k)             
}
\end{xy}
\end{displaymath}
where $L/k$ is a Galois extension and the 1-cocycle
\begin{eqnarray*}
\mathrm{Gal}(L/k)\longrightarrow \mathrm{Aut}(\mathrm{Grass}_L(m_1,...,m_r,V))
\end{eqnarray*}
 factors through $\mathrm{PGL}(V)$.

\section{Tits algebras}
We refer to \cite{PAS}, Section 3.1 for details (see also \cite{KNU}, p.376-379).
 Now let $G$ be a simply connected semi-simple algebraic group over the field $k$ and $P$ a parabolic subgroup. We denote by $\widetilde{G}$ and $\widetilde{P}$ their universal covers. For the center $\widetilde{Z}\subset \widetilde{G}$ let $\mathrm{Ch}:=\mathrm{Hom}(\widetilde{Z},\mathbb{G}_m)$ be the character group. Furthermore, let $R(\widetilde{G})$ and $R(\widetilde{P})$ be the associated representation rings. Recall from \cite{PAS},\S2 that there exits a finite free $\mathrm{Ch}$-homogeneous basis of $R(\widetilde{P})$ over $R(\widetilde{G})$. Furthermore, let $\chi\in \mathrm{Ch}$ and denote by $\mathrm{Rep}_{k}^{\chi}(\tilde{G})$ the full subcategory of $\mathrm{Rep}_k(\tilde{G})$ consisting of those $V$ such that $\tilde{Z}$ acts on $V$ by $\chi$. Now for a Galois-invariant $\chi\in \mathrm{Ch}$, choose a non-trivial representation $V_{\chi}\in \mathrm{Rep}_{k}^{\chi}(\tilde{G})$. Put $A_{\chi}=\mathrm{End}(V_{\chi})$. Then $A_{\chi}$ is an $k$-algebra equipped with a $G$-action by $k$-algebra automorphism. Using a 1-cocycle $\gamma\colon \mathrm{Gal}(k^s|k)\rightarrow G(k^s)$ one gets a new $\mathrm{Gal}(k^s|k)$-action on $A_{\chi}\otimes_k k^s$ and hence a twisted form $A_{\chi, \gamma}$.
In this way, one obtains the \emph{Tits map} (see \cite{PAS}, \S3 or \cite{KNU}, p.377) $\beta_{\gamma}\colon \mathrm{Ch}^{\Gamma}\rightarrow \mathrm{Br}(k)$ which is a group homomorphism and assigns to each character $\chi\in \mathrm{Ch}^{\Gamma}$ a central simple algebra $A_{\chi,\gamma}\in \mathrm{Br}(k)$, called \emph{Tits algebra}. 

\noindent
\begin{exam}[Type $A_n$]
	\textnormal{Let $G=\mathrm{SL}_1(A)$ where $A$ is a central simple algebra of degree $n+1$. Then $\bar{G}=\mathrm{PGL}_1(A)$ and $\mathrm{Ch}(Z)=\mathbb{Z}/(n+1)\mathbb{Z}$ with trivial $\mathrm{Gal}(k^s|k)$-action. For any $i=0,1,...,n$, consider the representation $p_i\colon \bar{G}\rightarrow \mathrm{GL}_1(\lambda^iA)$, where $\lambda^iA$ are external powers of $A$. In the split case, the $i$-th exterior power representation are known to be minimal representations (see \cite{KNU}). Hence $A^{\otimes i}$ are the  Tits algebras for $G$. } 
\end{exam}
\begin{exam}[Type $C_n$]
	\textnormal{Let $G=\mathrm{Sp}(A,\sigma)$ where $A$ is a central simple algebra of degree $2n$ with symplectic involution $\sigma$. Then $\bar{G}=\mathrm{PGSp}(A,\sigma)$ and $\mathrm{Ch}(Z)=\mathbb{Z}/2\mathbb{Z}=\{0,\chi\}$. The embedding $\bar{G}\rightarrow \mathrm{GL}_1(A)$ is in the split case a minimal representation. Hence $A$ is the Tits algebra.}
\end{exam}
\noindent
A complete list of the (minimal) Tits algebras for the simple $k$-split algebraic groups of classical type can be found for instance in \cite{KNU}, p. 378-379.

\section{AS-bundles on proper $k$-schemes} 
\begin{defi}
	\textnormal{A vector bundle $\mathcal{E}$ on a proper $k$-scheme $X$ is called \emph{pure of type} $\mathcal{W}$ if there is an indecomposable vector bundle $\mathcal{W}$ on $X\otimes_k \bar{k}$ such that $\mathcal{E}\otimes_k {\bar{k}}\simeq \mathcal{W}^{\oplus m}$. }
\end{defi}
\noindent
Recall from \cite{NO} the following definition.
\begin{defi}
	\textnormal{Let $X$ be a $k$-scheme. A vector bundle $\mathcal{E}$ on $X$ is called \emph{absolutely split} (\emph{separably split}) if it splits after base change as a direct sum of invertible sheaves on $X\otimes_k \bar{k}$ (resp. $X\otimes_k k^{sep}$). For an absolutely split vector bundle we shortly write \emph{AS-bundle}.}
\end{defi}
\begin{prop}[\cite{NO}, Proposition 4.2]
	Let $X$ be a proper $k$-scheme and $\mathcal{E}$ a vector bundle on $X$. Then $\mathcal{E}$ is absolutely split if and only if it is separably split.
\end{prop}
\begin{thm}[\cite{NO}, Theorem 4.6]
	Let $X$ be a proper $k$-scheme with $H^0(X,\mathcal{O}_X)=k$. Then the closed points of the Picard scheme $\mathrm{Pic}_{X/k}$ are in one-to-one correspondence with isomorphism classes of indecomposable $AS$-bundles on $X$.
\end{thm}
\noindent
In \cite{NO} the author classified all indecomposable $AS$-bundles on a proper $k$-scheme $X$ with cyclic Picard group. We want to generalize this result for the case $\mathrm{Pic}(X_s)\simeq \mathbb{Z}^{\oplus m}$. Theorem 4.5 below is interesting in its own right and can be used for instance to classify indecomposable $AS$-bundles on twisted flags. In the proof of Theorem 1.1 it is implicitely shown that twisted flags of classical type satisfy the assumption of Theorem 4.5. Note that the classification of indecomposable $AS$-bundles on the twisted flags under consideration is a vast generalization of the main theorem of \cite{BN}.

\noindent
Let us repeat the notations and facts from the introduction. $X$ is still a proper and geometrically integral $k$-scheme. From \cite{NO}, Proposition 3.4 it follows that $\mathrm{Pic}(X)$ is a subgroup of $\mathrm{Pic}(X_s)$. In particular, if $\mathrm{Pic}(X_s)\simeq \mathbb{Z}^{\oplus m}$, then $\mathrm{Pic}(X)\simeq r_1\mathbb{Z}\oplus \cdots \oplus r_m\mathbb{Z}$. 
Let us fix a basis $\mathcal{L}_1,...,\mathcal{L}_m$ of $\mathrm{Pic}(X_s)\simeq \mathbb{Z}^{\oplus m}$. 
Using a basis of $\mathrm{Pic}(X)$ and an easy computation from linear algebra involving matrices over the integers, one can show that there are line bundles $\mathcal{J}_i\in\mathrm{Pic}(X)$ 
satisfying $\mathcal{J}_i\otimes_k k^s\simeq \mathcal{L}_i^{\otimes c_i}$ for some integers $c_i\geq 1$. Now we choose the $\mathcal{J}_i$ such that the $c_i$ are minimal. According to \cite{NO}, Proposition 3.4 these line bundles $\mathcal{J}_i$ are unique up to isomorphism. 
Assume there are pure vector bundles $\mathcal{M}_i$ of type $\mathcal{L}_i\in \mathrm{Pic}(X_s)$. We know from \cite{NO}, Proposition 3.5 that the bundle $\mathcal{M}_i$ is unique up to isomorphism. We set $\mathcal{M}_{\mathcal{L}_i}:=\mathcal{M}_i$. It is easy to see that for any line bundle $\mathcal{L}_i^{\otimes a}\in \mathrm{Pic}(X_s)$ there is an indecomposable pure bundle of type $\mathcal{L}_i^{\otimes a}$. Indeed, let $s_i=\mathrm{rk}(\mathcal{M}_{\mathcal{L}_i})$ and consider $(\mathcal{L}_i^{\oplus s_i})^{\otimes a}\simeq (\mathcal{L}_i^{\otimes a})^{\oplus s_i^a}$. Then we get 
$\mathcal{M}_{\mathcal{L}_i}^{\otimes a}\otimes_k k^s\simeq (\mathcal{L}_i^{\oplus s_i})^{\otimes a}\simeq (\mathcal{L}_i^{\otimes a})^{\oplus s_i^a}$.

\noindent
Considering the Krull--Schmidt decomposition of $\mathcal{M}_{\mathcal{L}_i}^{\otimes a}$ and taking into account that all indecomposable direct summands are isomorphic (see \cite{NO}, proof of Proposition 3.6 and Remark 3.7), we get an, up to isomorphism, unique indecomposable vector bundle $\mathcal{M}_{\mathcal{L}_i^{\otimes a}}$ such that $\mathcal{M}_{\mathcal{L}_i^{\otimes a}}\otimes_k k^s\simeq (\mathcal{L}_i^{\otimes a})^{\oplus s_i(a)}$, where $s_i(a)=\mathrm{rank}(\mathcal{M}_{\mathcal{L}_i^{\otimes a}})$. 
Using Krull--Schmidt decomposition again, we can use our indecomposable vector bundles $\mathcal{M}_{\mathcal{L}_i^{\otimes a}}$ and take the tensor product of these to get an indecomposable vector bundle $\mathcal{M}_{(a_1,...,a_m)}$ of type $\mathcal{L}_1^{\otimes a_1}\otimes\cdots \otimes \mathcal{L}_m^{\otimes a_m}$. Again, the bundles $\mathcal{M}_{(a_1,...,a_m)}$ are unique up to isomorphism. 
\begin{thm} 
	Let $X$ be a proper, geometrically integral $k$-scheme with $\mathrm{Pic}(X_s)\simeq \mathbb{Z}^{\oplus m}$ and let $\mathcal{L}_1,...,\mathcal{L}_m\in \mathrm{Pic}(X_s)$ be the basis from above. Let $\mathcal{J}_i\in \mathrm{Pic}(X)\simeq \mathbb{Z}^{\oplus m}$ be up to isomorphism unique line bundles satisfying $\mathcal{J}_i\otimes_k k^s\simeq \mathcal{L}_i^{\otimes c_i}$ with $c_i$ being minimal. Assume there are indecomposable pure bundles $\mathcal{M}_{\mathcal{L}_i}$ of type $\mathcal{L}_i$. Then all indecomposable $AS$-bundles $\mathcal{E}$ are of the form
	\begin{eqnarray*}
		\mathcal{J}_1^{\otimes b_1}\otimes\cdots \otimes \mathcal{J}_m^{\otimes b_m} \otimes \mathcal{M}_{(a_1,...,a_m)}
	\end{eqnarray*}
	with unique $b_i\in\mathbb{Z}$ and $0\leq a_j\leq c_j-1$.
\end{thm} 
\begin{proof}
	Let $\mathcal{E}$ be an arbitrary, not necessarily indecomposable, $AS$-bundle and let $\pi:X\otimes_k k^s\rightarrow X$ the projection. By assumption, there are indecomposable pure vector bundles $\mathcal{M}_{\mathcal{L}_i}$ of type $\mathcal{L}_i$. Above we showed that there exist (up to isomorphism) unique indecomposable pure vector bundles of type $\mathcal{L}_i^{\otimes a}$ for all $a\in \mathbb{Z}$. 
	 Let $d=\mathrm{lcm}(\mathrm{rk}(\mathcal{M}_{(a_1,...,a_m)})$,  $0\leq a_j\leq c_j-1$, be the least common multiple and consider the vector bundle $\pi^*(\mathcal{E}^{\oplus d})$. Since $\mathcal{E}$ is an $AS$-bundle, the vector bundle $\mathcal{E}^{\oplus d}$ is an $AS$-bundle, too. Therefore $\pi^*(\mathcal{E}^{\oplus d})$ decomposes into a direct sum of invertible sheaves. Below we give the proof for $m=2$ to simplify the notation. So after reordering $(\mathrm{mod} \ r_1,\mathrm{mod} \ r_2)$ in lexicographical order, we find that $\pi^*(\mathcal{E}^{\oplus d})$ is isomorphic to the bundle
	\begin{eqnarray*}
		\left(\bigoplus(\mathcal{L}_1^{\otimes {s_{i_{0}}^{(1)}\cdot c_1+0}}\otimes \mathcal{L}_2^{\otimes {t_{i_{0}}^{(1)}\cdot c_2}+0})^{\oplus d}\right)\oplus \left(\bigoplus(\mathcal{L}_1^{\otimes {s_{i_{1}}^{(1)}\cdot c_1+0}}\otimes \mathcal{L}_2^{\otimes {t_{i_{1}}^{(1)}\cdot c_2}+1})^{\oplus d}\right)        \oplus\cdots \\
		\oplus \left(\bigoplus(\mathcal{L}_1^{\otimes {s_{i_{(c_2-1)}}^{(1)}\cdot c_1+0}}\otimes \mathcal{L}_2^{\otimes {t_{i_{(c_2-1)}}^{(1)}\cdot c_2}+(r_2-1)})^{\oplus d}\right)\oplus \cdots\\
		\oplus \left(\bigoplus(\mathcal{L}_1^{\otimes {s_{i_{(c_2-1)}}^{(c_1-1)}\cdot c_1+(c_1-1)}}\otimes \mathcal{L}_2^{\otimes {t_{i_{(c_2-1)}}^{(c_1-1)}\cdot c_2}+(c_2-1)})^{\oplus d}\right)
	\end{eqnarray*}
	By definition of $d$, there are $h_{(p,q)}$ such that $h_{(p,q)}\cdot\mathrm{rk}(\mathcal{M}_{(p,q)})=d$ for $0\leq p\leq c_1-1$ and $0\leq q\leq c_2-1$. Furthermore, the sheaves $\mathcal{M}_{(p,q)}$ satisfy
	\begin{eqnarray*}
	 \pi^*\mathcal{M}_{(p,q)}\simeq(\mathcal{L}_1^{\otimes p}\otimes \mathcal{L}_2^{\otimes q})^{\oplus d_{(p,q)}},  
	 \end{eqnarray*}
 where $d_{(p,q)}=\mathrm{rk}(\mathcal{M}_{(p,q)})$.
 Now for the direct summands 
	\begin{eqnarray*}
		\left(\mathcal{L}_1^{\otimes {s_{i_{m}}^{(l)}\cdot c_1+p}}\otimes \mathcal{L}_2^{\otimes {t_{i_{m}}^{(l)}\cdot c_2}+q}\right )^{\oplus d}
		\end{eqnarray*}	
	where $0\leq p\leq c_1-1$ and $0\leq q\leq c_2-1$, we have 
	\begin{eqnarray*}
	\left (\left(\mathcal{L}_1^{\otimes {s_{i_{m}}^{(l)}\cdot c_1+p}}\otimes \mathcal{L}_2^{\otimes {t_{i_{m}}^{(l)}\cdot c_2}+q}\right )^{\oplus d_{(p,q)}}\right )^{\oplus h_{(p,q)}}.	
	\end{eqnarray*} Considering the vector bundle $(\mathcal{J}_1^{\otimes  {s_{i_{m}}^{(l)}}}\otimes\mathcal{J}_2^{\otimes  {t_{i_{m}}^{(l)}}}\otimes \mathcal{M}_{(p,q)})^{\oplus h_{(p,q)}}$ on $X$, we find
	\begin{eqnarray*}
		\pi^*\left(\mathcal{J}_1^{\otimes  {s_{i_{m}}^{(l)}}}\otimes\mathcal{J}_2^{\otimes  {t_{i_{m}}^{(l)}}}\otimes \mathcal{M}_{(p,q)}\right)^{\oplus h_{(p,q)}}\simeq \left(\mathcal{L}_1^{\otimes {s_{i_{m}}^{(l)}\cdot c_1+p}}\otimes \mathcal{L}_2^{\otimes {t_{i_{m}}^{(l)}\cdot c_2}+q}\right )^{\oplus d}
	\end{eqnarray*} Now consider the vector bundle 
	\begin{eqnarray*}
		\left(\bigoplus(\mathcal{J}_1^{\otimes s_{i_{0}}^{(1)}}\otimes \mathcal{J}_2^{t_{i_{0}}^{(1)}}\otimes \mathcal{M}_{(0,0)})^{\oplus d}\right)\oplus \left(\bigoplus(\mathcal{J}_1^{\otimes s_{i_{1}}^{(1)}}\otimes \mathcal{J}_2^{t_{i_{1}}^{(1)}}\otimes \mathcal{M}_{(0,1)})^{\oplus h_{(0,1)}}\right)    \oplus\cdots\\
		\oplus \left(\bigoplus(\mathcal{J}_1^{\otimes s_{i_{(c_1-1)}}^{(1)}}\otimes \mathcal{J}_2^{t_{i_{(c_1-1)}}^{(1)}}\otimes \mathcal{M}_{(0,c_2-1)})^{\oplus h_{(0,c_2-1)}}\right)\oplus\cdots\\
		\oplus \left(\bigoplus(\mathcal{J}_1^{\otimes s_{i_{(c_2-1)}}^{(c_1-1)}}\otimes \mathcal{J}_2^{t_{i_{(c_2-1)}}^{(c_1-1)}}\otimes \mathcal{M}_{(c_1-1,c_2-1)})^{\oplus h_{(c_1-1,c_2-1)}}\right)
	\end{eqnarray*} 
which is denoted by $\mathcal{R}$. We immediately see that $\pi^*\mathcal{R}\simeq \pi^*(\mathcal{E}^{\oplus d})$. Applying \cite{NO}, Proposition 3.4 implies that $\mathcal{E}^{\oplus d}$ is isomorphic to $\mathcal{R}$. Because Krull--Schmidt Theorem holds for vector bundles on $X$, we conclude that $\mathcal{E}$ is isomorphic to the direct sum of vector bundles of the form 
\begin{eqnarray*}
	\mathcal{J}_1^{\otimes b_1}\otimes\mathcal{J}_2^{\otimes b_2} \otimes \mathcal{M}_{(a_1,a_2)}
	\end{eqnarray*}
with unique $b_i\in\mathbb{Z}$ and $0\leq a_j\leq c_j-1$. Furthermore, since all these bundles are indecomposable by definition, we finally get that all the indecomposable $AS$-bundles have the desired form. This completes the proof.
\end{proof}
\begin{lem}
	Let $X$ be a proper and geometrically integral $k$-scheme and $\mathcal{L},\mathcal{L}_1$ and $\mathcal{L}_2$ line bundles.
	\begin{itemize}
		\item[(i)] If $\mathcal{M}$ is pure of type $\mathcal{L}$, then $\mathrm{End}(\mathcal{M})$ is a central simple $k$-algebra.
		\item[(ii)] There is an (up to isomorphism) unique indecomposable pure vector bundle $\mathcal{M}_{\mathcal{L}}$ of type $\mathcal{L}$.
		\item[(iii)] Let $\mathcal{M}_{\mathcal{L}_1}$ and $\mathcal{M}_{\mathcal{L}_2}$ be pure vector bundles of type $\mathcal{L}_1$ resp. $\mathcal{L}_2$. Then $\mathrm{End}(\mathcal{M}_{\mathcal{L}_1})\otimes \mathrm{End}(\mathcal{M}_{\mathcal{L}_2})$ is Brauer-equivalent to $\mathrm{End}(\mathcal{M}_{\mathcal{L}_1\otimes \mathcal{L}_2})$.
	\end{itemize}
\begin{proof}
Since $X$ is geometrically integral, we have $H^0(X_s,\mathcal{O}_{X_s})\simeq k^s$ (see \cite{NO}, Proposition 4.2). This implies $\mathcal{M}\otimes_k k^s\simeq \mathcal{L}^{\oplus r}$ and therefore
	\begin{eqnarray*}
		\mathrm{End}(\mathcal{M})\otimes_k k^s\simeq \mathrm{Mat}_r(k^s).
	\end{eqnarray*}
Then \cite{KNU}, Theorem (1.1) shows that  $\mathrm{End}(\mathcal{M})$ must be central simple over $k$. This shows (i). Assertion (ii) follows directly from \cite{JK}, Lemma 8. It remains to show (iii). For this, let $\Delta\colon X\rightarrow X\times X$ be the diagonal and $\pi_i\colon X\times X\rightarrow X$, $i=1,2$, the two projections. For $\mathcal{L}:=\pi_1^*\mathcal{L}_1\otimes \pi_2^*\mathcal{L}_2$ one has $\Delta^*\mathcal{L}\simeq \mathcal{L}_1\otimes \mathcal{L}_2$. From \cite{JK}, Corollary 11 it follows $\mathcal{M}_{\Delta^*\mathcal{L}}\simeq \Delta^*\mathcal{M}_{\mathcal{L}}$. Moreover, the proof of Corollary 11 in \emph{loc.cite} shows $\mathrm{End}(\mathcal{M}_{\mathcal{L}})\simeq \mathrm{End}(\Delta^*\mathcal{M}_{\mathcal{L}})\simeq \mathrm{End}(\mathcal{M}_{\mathcal{L}_1\otimes \mathcal{L}_2})$. Now \cite{JK}, p.14 explains that $\mathrm{End}(\mathcal{L}_1)\otimes \mathrm{End}(\mathcal{L}_2)$ is Brauer-equivalent to $\mathrm{End}(\mathcal{M}_{\mathcal{L}_1\otimes \mathcal{L}_2})$. This completes the proof.
\end{proof}		
		
\end{lem} 
\begin{prop}
	Let $\mathcal{E}=\mathcal{J}_1^{\otimes b_1}\otimes\cdots \otimes \mathcal{J}_m^{\otimes b_m} \otimes \mathcal{M}_{(a_1,...,a_m)}$ be an indecomposable $AS$-bundle as in Theorem 4.5. Then $\mathrm{End}(\mathcal{E})$ is a central simple $k$-algebra.
\end{prop}
\begin{proof}
Since $\mathrm{End}(\mathcal{E})\simeq \mathrm{End}(\mathcal{M}_{(a_1,...,a_m)})$, the assertion follows from Lemma 4.6 (i).	
\end{proof}
\begin{prop}
	Let $\mathcal{E}=\mathcal{J}_1^{\otimes b_1}\otimes\cdots \otimes \mathcal{J}_m^{\otimes b_m} \otimes \mathcal{M}_{(a_1,...,a_m)}$ and $\mathcal{E}'=\mathcal{J}_1^{\otimes b'_1}\otimes\cdots \otimes \mathcal{J}_m^{\otimes b'_m} \otimes \mathcal{M}_{(a'_1,...,a'_m)}$ be two indecopmosable $AS$-bundles as in Theorem 4.5. Then there is a unique positive integer $s$ such that 
\begin{eqnarray*}
	(\mathcal{E}\otimes \mathcal{E}')\simeq  \mathcal{J}_1^{\otimes (b_1+b_1')}\otimes\cdots \otimes \mathcal{J}_m^{\otimes (b_m+b_m')} \otimes \mathcal{M}_{((a_1+a_1'),...,(a_m+a_m'))}^{\oplus s}.
	\end{eqnarray*}	
Moreover, $\mathrm{End}(\mathcal{E}\otimes \mathcal{E}')$ is Brauer-equivalent to $\mathrm{End}(\mathcal{E})\otimes \mathrm{End}(\mathcal{E}')$ in $\mathrm{Br}(k)$.
	\end{prop}
\begin{proof}
Note that the $AS$-bundle $\mathcal{M}_{(a_1,...,a_m)}\otimes \mathcal{M}_{(a'_1,...,a'_m)}$ is a pure vector bundle of type $\mathcal{L}_1^{\otimes (a_1+a_1')}\otimes\cdots \otimes \mathcal{L}_m^{\otimes (a_m+a_m')}$.  The first assertion follows from Lemma 4.6 (ii).	Now we want to prove that $\mathrm{End}(\mathcal{E}\otimes \mathcal{E}')$ is Brauer-equivalent to $\mathrm{End}(\mathcal{E})\otimes \mathrm{End}(\mathcal{E}')$. As mentioned in the proof of Proposition 4.7, we have $\mathrm{End}(\mathcal{E})\simeq \mathrm{End}(\mathcal{M}_{(a_1,...,a_m)})$ and $\mathrm{End}(\mathcal{E}')\simeq \mathrm{End}(\mathcal{M}_{(a'_1,...,a'_m)})$. 
We conclude with Lemma 4.6 (iii). 
\end{proof}

\section{Proof of Theorem 1.1}
\begin{proof}
According to Theorem 4.4, the closed points of $\mathrm{Pic}_{(X/k)(et)}$ are in one-to-one correspondence with isomorphism classes of indecomposable $AS$-bundles. Since $X$ is proper over $k$, we have $\mathrm{Pi}_{(X/k)(\mathrm{et})}\simeq \mathrm{Pic}_{(X/k)(\mathrm{fppf})}$ as abelian groups. And because $\mathrm{Pic}(X_s)\simeq \mathbb{Z}^{\oplus m}$, we can use Theorem 4.5 to obtain a classification of all indecomposable $AS$-bundles on $X$. In particular, the $k$-rational points of $\mathrm{Pic}_{(X/k)(et)}$ correspond to indecomposable $AS$-bundles. As mentioned in the introduction, being pure vector bundle of type $\mathcal{L}\in \mathrm{Pic}(X_s)$ is equvalent to $\mathcal{L}\in \mathrm{Pic}_{\Gamma}(X_s)$. By assumption, there are pure vector bundles of type $\mathcal{L}_i$. This actually implies that the basis $\mathcal{L}_1,...,\mathcal{L}_m$ of $\mathrm{Pic}(X_s)$ consists of Galois invariant line bundles. Therefore, the indecomposable $AS$-bundles are in one-to-one correspondence with the $k$-rational points of $\mathrm{Pic}_{(X/k)(et)}$. 
Now consider the exact sequence from the introduction and specialize it to the case $T=S=\mathrm{Spec}(k)$. We get the following exact sequence
\begin{eqnarray*}
	0\longrightarrow \mathrm{Pic}(X)\longrightarrow \mathrm{Pic}_{(X/S)(\mathrm{fppf})}(k)\stackrel{\delta}{\longrightarrow}\mathrm{Br}(k)\longrightarrow \mathrm{Br}'(X)
\end{eqnarray*}	
where $\delta(\mathcal{E})=[\mathrm{End}(\mathcal{E})]\in \mathrm{Br}(k)$ for an indecomposable $AS$-bundle $\mathcal{E}$.
Finally, use Theorem 4.5, Lemma 4.6, Lemma 4.8, and \cite{PAS}, Lemma 3.4 
to conclude that $\mathrm{Am}(X)$ is indeed generated by the set $D$ which consists of the classes $\mathrm{End}(\mathcal{M}_{e_i})$, where $i\in\{1,...m\}$ satisfying $c_i\geq 2$. The above arguments also show $\#D\leq \mathrm{rank}(\mathrm{Pic}(X))$. It remains to show that if $X$ is geometrically rational, one has $\mathrm{Br}(X)\simeq \mathrm{Br}(k)/\mathrm{Am}(X)$. For this, we consider the Hochschild--Serre spectral sequence $H^p(k,H^q(X_s,\mathbb{G}_m))\Rightarrow H^{p+q}(X,\mathbb{G}_m)$. Since $X$ is geometrically rational, we have $\mathrm{Br}(X_s)=0$. The spectral sequence then yields
\begin{eqnarray*}
	\mathrm{Pic}_{\Gamma}(X_s)\longrightarrow \mathrm{Br}(k)\longrightarrow \mathrm{Br}(X)\longrightarrow H^1(k,\mathrm{Pic}(X_s))
	\end{eqnarray*}
where $\Gamma$ denotes the absolute Galois group. By assumption, there are pure vector bundles of type $\mathcal{L}_i$. This actually implies that the basis $\mathcal{L}_1,...,\mathcal{L}_m$ of $\mathrm{Pic}(X_s)$ consists of Galois invariant line bundles. Therefore $\mathrm{Pic}(X_s)$ is a permutation $\Gamma$-module and hence $H^1(k,\mathrm{Pic}(X_s))=0$. The above exact sequence yields $\mathrm{Br}(X)\simeq \mathrm{Br}(k)/\mathrm{Am}(X)$. This completes the proof.
\end{proof}

\noindent

\noindent(\textbf{proof of Corollary 1.2}):\\
\noindent
Recall that for any flag $G/P$ of classical type associated to some semisimple simply connected and $k$-split $G$, one has $\mathrm{Pic}(G/P)\simeq \mathbb{Z}^{\oplus m}$. Now, let $G$ be an semisimple simply connected algebraic group over $k$. Consider a twisted flag $X:={_\gamma}(G/P)$. Note that after base change to a separable closure $k^s$ we have $X_s\simeq G_s/P_s$. 
\noindent
Denote by $\mathcal{A}_1,...,\mathcal{A}_m$ the basis of $\mathrm{Pic}(X_s)$ given in \cite{METI}, p.55f. From \cite{METI}, p.37 we conclude that $\mathrm{End}(\mathcal{B}_{\mathcal{A}_i})$ is a Tits algebra. Using these $\mathcal{A}_i$, we can construct a basis $\mathcal{L}_1,...,\mathcal{L}_m$ of $\mathrm{Pic}(X_s)\simeq \mathbb{Z}^{\oplus m}$ 
consisting of Galois-invariant line bundles. Let $\mathcal{J}_1,...\mathcal{J}_m$ be the line bundles in $\mathrm{Pic}(X)$ satisfying $\mathcal{J}_j\otimes_k k^s \simeq \mathcal{L}_j^{\otimes c_j}$ with minimal $c_i$. 
Notice that $H^0(X_s,\mathcal{O}_{X_s})\simeq k^s$. According to Theorem 4.4, the closed points of $\mathrm{Pic}_{(X/k)(et)}$ are in one-to-one correspondence with isomorphism classes of indecomposable $AS$-bundles. Since $X$ is proper over $k$, we have $\mathrm{Pic}_{(X/k)(\mathrm{et})}\simeq \mathrm{Pic}_{(X/k)(\mathrm{fppf})}$ as abelian groups. And because $\mathrm{Pic}(G_s/P_s)\simeq \mathbb{Z}^{\oplus m}$ and since we have a basis $\mathcal{L}_1,...,\mathcal{L}_m$ of $\mathrm{Pic}(X_s)$ consisting of Galois-invariant line bundles, it follows from \cite{NO}, Proposition 4.3 that there are (unique) pure indecomposable vector bundles $\mathcal{W}_{\mathcal{L}_i}$ of type $\mathcal{L}_j$. Moreover, we have that the indecomposable $AS$-bundles are in one-to-one correspondence with the $k$-rational points of $\mathrm{Pic}_{(X/k)(et)}$. Now we can use Theorem 4.5 from above to obtain a classification of all indecomposable $AS$-bundles on the twisted flags of classical type. By construction, and by applying Lemma 4.6 (iii) we conclude that there exist integers $a_1,...,a_m$ such that $\mathrm{End}(\mathcal{W}_{\mathcal{L}_i})$ is Brauer equivalent to $\mathrm{End}(\mathcal{B}_{\mathcal{A}_1})^{\otimes a_1}\otimes\cdots \otimes \mathrm{End}(\mathcal{B}_{\mathcal{A}_m})^{\otimes a_m}$.

Now consider the exact sequence from the introduction and specialize it to the case $T=S=\mathrm{Spec}(k)$. We get the following exact sequence
\begin{eqnarray*}
	0\longrightarrow \mathrm{Pic}(X)\longrightarrow \mathrm{Pic}_{(X/S)(\mathrm{fppf})}(k)\stackrel{\delta}{\longrightarrow}\mathrm{Br}(k)\longrightarrow \mathrm{Br}'(X)
\end{eqnarray*}	
where $\delta(\mathcal{E})=[\mathrm{End}(\mathcal{E})]\in \mathrm{Br}(k)$ for an indecomposable $AS$-bundle $\mathcal{E}$.
Finally, use Theorem 4.5, Lemma 4.6, Lemma 4.8 and \cite{PAS}, Lemma 3.4 
to conclude that $\mathrm{Am}(X)$ is indeed generated by the subset $D$ which consists of classes of Tits algebras coming from a basis of $\mathrm{Ch}(P_s)^{\Gamma}$ which is given by some of the $\mathcal{L}_i$. For these $\mathcal{L}_i$ one has $c_i\geq 2$. The above arguments also show $\#D\leq \mathrm{rank}(\mathrm{Pic}(X_s))$. This completes the proof.\\

\noindent
(\textbf{proof of Corollary 1.3})

\noindent
We sketch the proof for $X=X_1\times X_2$. Let $X_1={_{\gamma_1}}(G_1/P_1)$ and $X_2={_{\gamma_2}}(G_2/P_2)$ and notice that $\mathrm{Pic}(X_s)\simeq \mathrm{Pic}((X_1)_s)\times \mathrm{Pic}((X_2)_s)$ since $(X_i)_s$ are rational over $k^s$. Therefore $\mathrm{Pic}(X_s)\simeq \mathbb{Z}^{\oplus m_1}\oplus \mathbb{Z}^{\oplus m_2}$. This implies $\mathrm{Pic}(X)\simeq \mathbb{Z}^{\oplus m_1+m_2}$. Now let $\mathcal{L}_1,...,\mathcal{L}_{m_1}$ be the generators of $\mathrm{Pic}((X_1)_s)$ and $\mathcal{K}_1,...,\mathcal{K}_{m_2}$ the generators of $\mathrm{Pic}((X_2)_s)$. Now proceed as in the proof of Theorem 1.1  and use \cite{MPW}, Corollary 2.3 to obtain the desired assertion.
\begin{rema}
\textnormal{We wonder whether the Amitsur subgroup is a complete birational invariant for twisted flags of classical type. In the special case of  Brauer--Severi varieties, this problem is known as the \emph{Amitsur conjecture for central simple algebras} \cite{AM}. This conjecture is still open in general. For details and results in this direction we refer to \cite{GSS} and \cite{KR1} and references therein.} 	
\end{rema}
\begin{exam}
	\textnormal{Let $X$ be a Brauer--Severi variety corresponding to the central simple algebra $A$. Then $X_s\simeq \mathbb{P}^n$ and $\mathrm{Pic}(X_s)=\mathbb{Z}$ is generated by the ample line bundle $\mathcal{O}_{\mathbb{P}^n}(1)$. In Section 4 we showed that there is a (up to isomorphism) unique $AS$-bundle $\mathcal{M}_1$ of type $\mathcal{O}_{\mathbb{P}^n}(1)$. The proof of Theorem 1.1 actually shows that $[\mathrm{End}(\mathcal{M}_1)]$ generates $\mathrm{Am}(X)$. It is well known that $[\mathrm{End}(\mathcal{M}_1)]=[A]^{-1}$ (see for instance \cite{PAS}, p.571). Therefore, $\mathrm{Am}(X)=\langle [A]\rangle$. This gives back the classical result due to Ch\^{a}telet which is mentioned in the introduction.}
\end{exam}

\begin{exam}
	\textnormal{Let $X=\mathrm{BS}(d,A)$ be a generalized Brauer--Severi corresponding to a central simple algebra $A$ of degree $n$. We have $X_s\simeq \mathrm{Grass}(d,n)$ and $\mathrm{Pic}(X_s)=\mathbb{Z}$ is generated by the ample line bundle $\mathcal{O}(1)=\mathrm{det}(\mathcal{Q})$ where $\mathcal{Q}$ is the universal quotient bundle on $\mathrm{Grass}(d,n)$. As explained in \cite{NO}, p.16 there is a (up to isomorphism) unique $AS$-bundle $\mathcal{N}$ of type $\mathcal{O}(1)$. Moreover, it is well known that $\mathrm{End}(\mathcal{N})$ is Brauer-equivalent to $A^{\otimes -d}$ (see for instance \cite{PAS}, p.572). Hence $\mathrm{Am}(X)=\langle [A^{\otimes d}]\rangle$. This gives back \cite{BLS}, Theorem 7.  }
\end{exam}

\section{Application to noncommutative motives}
	
\noindent(\textbf{proof of Theorem 1.4 and Corollary 1.6})

\noindent
If $X$ and $Y$ are birational, we conclude from \cite{LIT}, Proposition 2.10 that $\mathrm{Am}(X)=\mathrm{Am}(Y)$ in $\mathrm{Br}(k)$. According to Theorem 1.1, both Amitsur subgroups $\mathrm{Am}(X)$ and $\mathrm{Am}(Y)$ are generated by certain Tits algebras of the algebraic groups involved. Since the Amitsur subgroup is a finitely generated torsion abelian subgroup of $\mathrm{Br}(k)$, we conclude with the fundamental theorem of finitely generated abelian groups that 
\begin{eqnarray*}
	\mathrm{Am}(X)\simeq \mathbb{Z}/p_1^{r_1}\mathbb{Z}\times\cdots \times \mathbb{Z}/p_s^{r_s}\mathbb{Z}, \ \mathrm{Am}(Y)\simeq \mathbb{Z}/q_1^{v_1}\mathbb{Z}\times\cdots \times \mathbb{Z}/q_t^{v_t}\mathbb{Z}
	\end{eqnarray*}
with uniquely determined $p_i^{r_i}$ and $q_j^{v_j}$ where $p_i$ and $q_j$ are prime numbers. Since $\mathrm{Am}(X)=\mathrm{Am}(Y)$, we have $s=t$ and isomorphic factors up to permutation. Without loss of generality we assume $p_i=q_i$ and therefore $r_i=v_i$. Let $a_i$ be a generator of $\mathbb{Z}/p_i^{r_i}\mathbb{Z}$ and $b_i$ a generator of $\mathbb{Z}/q_i^{v_i}\mathbb{Z}$. Denote by $e_i=(0,...,a_i,0,...,0)$ $i=1,...,s$ a set of generators of $\mathbb{Z}/p_1^{r_1}\mathbb{Z}\times\cdots \times \mathbb{Z}/p_s^{r_s}\mathbb{Z}$ and by $f_i=(0,...,b_i,0,...,0)$ $i=1,...,t$ a set of generators of $\mathbb{Z}/q_1^{v_1}\mathbb{Z}\times\cdots \times \mathbb{Z}/q_t^{v_t}\mathbb{Z}$. The corresponding central simple algebras are denoted by $A_{e_i}$ and $B_{f_i}$ respectively. By definition, we have $[A_{e_i}]=[B_{f_i}^{\otimes n_i}]$ for a unique positive integer $n_i$. Note that from \cite{TAZ}, (2.18) it follows 
\begin{eqnarray*}
	\bigoplus_{i=1}^sU(A_{e_i})\simeq \bigoplus_{i=1}^tU(B_{f_i}^{\otimes n_i}).
	\end{eqnarray*}
Now let $A_g$ be the central simple division algebra corresponding to $g\in\mathrm{Am}(X)$. Analogously, let $B_h$ be the central simple division algebra corresponding to $h\in\mathrm{Am}(Y)$. Now \cite{TAZ}, Theorem Theorem 2.19 implies 
\begin{eqnarray*}
	M_X:=\bigoplus _{g\in \mathrm{Am}(X)}U(A_{g})\  \simeq \bigoplus_{h\in\mathrm{Am}(Y)}U(B_{h})=:M_Y.
	\end{eqnarray*}
\noindent
\textbf{Claim:}
\noindent
Let $G$, $P$ and $\gamma$ be as in Corollary 1.2 and let $\rho_1,...,\rho_n$ be a $\mathrm{Ch}$-homogeneous basis of $R(P)$ over $R(G)$, where $R(P)$ and $R(G)$ denote the corresponding representation rings. Let $A_{\chi(i),\gamma}$ be the Tits central simple algebras associated to $\rho_i$ and let $\mathrm{Ti}_{\rho_1,...,\rho_n}(X):=\langle A_{\chi(1),\gamma},...,A_{\chi(n),\gamma} \rangle$ be the subgroup generated by these Tits algebras. Denote by 
\begin{eqnarray*}
M\mathrm{Ti}_{\rho_1,...,\rho_n}(X):=\bigoplus_{f\in \mathrm{Ti}_{\rho_1,...,\rho_n}(X)}U(A_f),
\end{eqnarray*}
where $A_f$ are the central simple algebras corresponding to $f$. Then $U(X)\oplus N'=M\mathrm{Ti}_{\rho_1,...,\rho_n}(X)$ with $N'\in\mathrm{CSA}(k)^{\oplus}$.
\begin{proof}
By \cite{TA1S}, Theorem 2.1 we conclude that $U(X)=U( A_{\chi(1),\gamma})\oplus\cdots \oplus U( A_{\chi(n),\gamma})$. Obviously, $U(X)$ is a direct summand of $M\mathrm{Ti}_{\rho_1,...,\rho_n}(X)$ and $N'\in\mathrm{CSA}(k)^{\oplus}$ by construction.	
\end{proof}
\noindent
Note that $\mathrm{Am}(X)$ is a subgroup of $\mathrm{Ti}_{\rho_1,...,\rho_n}(X)$. It is clear that there exists a $N\in \mathrm{CSA}(k)^{\oplus}$ such that $M_X\oplus N=M\mathrm{Ti}_{\rho_1,...,\rho_n}(X)$. Now use the claim to conclude that $U(X)$ is a direct summand of $M\mathrm{Ti}_{\rho_1,...,\rho_n}(X)$. Hence $M_X\oplus N=U(X)\oplus N'$ with $N'\in \mathrm{CSA}(k)^{\oplus}$. By the same argument we obtain $M_Y\oplus Q=U(Y)\oplus Q'$. This completes the proof of Theorem 1.4 and Corollary 1.6.
\begin{exam}
	\textnormal{Let $X$ and $Y$ be Brauer--Severi varieties corresponding to central simple algebras $A$ and $B$. Then $\mathrm{Am}(X)=\langle A\rangle $ and $\mathrm{Am}(Y)=\langle B\rangle $. Now if $X$ and $Y$ are birational, then $\langle A\rangle=\langle B\rangle $ according to a theorem of Amitsur \cite{AM}. Since $\mathrm{Am}(X)$ and $\mathrm{Am}(Y)$ are cyclic of order $\mathrm{per}(A)=\mathrm{per}(B):=m$, we conclude from \cite{TAZ}, Theorem 3.20
	\begin{eqnarray*}
		M_X=U(k)\oplus U(A)\oplus \cdots \oplus U(A^{\otimes m-1})\simeq U(k)\oplus U(B)\oplus \cdots \oplus U(B^{\otimes m-1})=M_Y
		\end{eqnarray*}	
		Since $m\cdot r=\mathrm{deg}(A)=\mathrm{deg}(B)$, we can use \cite{TAZ}, Theorem 2.19 to conclude that $M_X^{\oplus r}\simeq U(X)$ and $M_Y^{\oplus r}\simeq U(Y)$. So in the case of Brauer--Severi varieties we have $N=M_X^{\oplus (r-1)}$, $N'=0$, $Q=M_Y^{\oplus (r-1)}$ and $Q'=0$.}
\end{exam}
\begin{exam}
	\textnormal{Let $X=\mathrm{BS}(d,A)$ and $Y=\mathrm{BS}(d,B)$ be generalized Brauer--Severi varieties corresponding to central simple algebras $A$ and $B$. Then $\mathrm{Am}(X)=\langle A^{\otimes d}\rangle $ and $\mathrm{Am}(Y)=\langle B^{\otimes d}\rangle $ (see \cite{BLS}, Theorem 7). If $X$ is birational to $Y$, then $\mathrm{Am}(X)=\mathrm{Am}(Y)$. Notice that $\mathrm{Am}(X)$ and $\mathrm{Am}(Y)$ are cyclic of order $m=\mathrm{per}(A)/\mathrm{gcd}(d,\mathrm{per}(A))$. According to \cite{TAZ}, Theorem 3.20 one has $M_X\simeq M_Y$, where 
	\begin{eqnarray*}
		M_X=U(k)\oplus U(A^{\otimes d})\oplus \cdots \oplus U(A^{\otimes dm-d}),\\
		M_Y=U(k)\oplus U(B^{\otimes d})\oplus \cdots \oplus U(B^{\otimes dm-d}).
	\end{eqnarray*}		
 One can use \cite{TAZ}, Theorem 2.19 and 3.18 to conclude that $M_X$ is a direct summand of $U(X)$ and $M_Y$ a direct summand of $U(Y)$. Again, we have $N'=0=Q'$.}
\end{exam}
In particular, for birational (generalized) Brauer--Severi varieties $X$ and $Y$ one has $U(X)\oplus Q\simeq U(Y)\oplus N$. Not that if $X$ and $Y$ are Brauer--Severi, then $N\simeq Q$ and \cite{TA2S}, Proposition 4.5 implies $U(X)\simeq U(Y)$. In this way we get back \cite{TAZ}, Proposition 3.15. 
Using the theory of semiorthogonal decompositions one can try to generalize Corollary 1.6 to arbitrary proper and geometrically integral $k$-schemes. We recall the definition of exceptional object and semiorthogonal decomposition.

 Let $\mathcal{D}$ be a triangulated category and $\mathcal{C}$ a triangulated subcategory. The subcategory $\mathcal{C}$ is called \emph{thick} if it is closed under isomorphisms and direct summands. For a subset $A$ of objects of $\mathcal{D}$ we denote by $\langle A\rangle$ the smallest full thick subcategory of $\mathcal{D}$ containing the elements of $A$. 
Furthermore, we define $A^{\perp}$ to be the subcategory of $\mathcal{D}$ consisting of all objects $M$ such that $\mathrm{Hom}_{\mathcal{D}}(E[i],M)=0$ for all $i\in \mathbb{Z}$ and all elements $E$ of $A$. We say that $A$ \emph{generates} $\mathcal{D}$ if $A^{\perp}=0$. Now assume $\mathcal{D}$ admits arbitrary direct sums. An object $B$ is called \emph{compact} if $\mathrm{Hom}_{\mathcal{D}}(B,-)$ commutes with direct sums. Denoting by $\mathcal{D}^c$ the subcategory of compact objects we say that $\mathcal{D}$ is \emph{compactly generated} if the objects of $\mathcal{D}^c$ generate $\mathcal{D}$. One has the following important theorem (see \cite{BVS}, Theorem 2.1.2).
\begin{thm}
	Let $\mathcal{D}$ be a compactly generated triangulated category. Then a set of objects $A\subset \mathcal{D}^c$ generates $\mathcal{D}$ if and only if $\langle A\rangle=\mathcal{D}^c$.  
\end{thm}
For a smooth projective scheme $X$ over $k$, we denote by $D(\mathrm{Qcoh}(X))$ the derived category of quasicoherent sheaves on $X$. The bounded derived category of coherent sheaves is denoted by $D^b(X)$. Note that $D(\mathrm{Qcoh}(X))$ is compactly generated with compact objects being all of $D^b(X)$. For details on generating see \cite{BVS}.

\begin{defi}
	\textnormal{Let $A$ be a division algebra over $k$, not necessarily central. An object $\mathcal{E}\in D^b(X)$ is called \emph{w-exceptional} if $\mathrm{End}(\mathcal{E})=A$ and $\mathrm{Hom}(\mathcal{E},\mathcal{E}[r])=0$ for $r\neq 0$. If $A=k$ the object is called \emph{exceptional}. If $A$ is a separable $k$-algebra, the object $\mathcal{E}$ is called \emph{separable-exceptional}. } 
\end{defi}
\begin{defi}
	\textnormal{A totally ordered set $\{\mathcal{E}_1,...,\mathcal{E}_n\}$ of w-exceptional (resp. separable-exceptional) objects on $X$ is called an \emph{w-exceptional collection} (resp. \emph{separable-exceptional collection}) if $\mathrm{Hom}(\mathcal{E}_i,\mathcal{E}_j[r])=0$ for all integers $r$ whenever $i>j$. An w-exceptional (resp. separable-exceptional) collection is \emph{full} if $\langle\{\mathcal{E}_1,...,\mathcal{E}_n\}\rangle=D^b(X)$ and \emph{strong} if $\mathrm{Hom}(\mathcal{E}_i,\mathcal{E}_j[r])=0$ whenever $r\neq 0$. If the set $\{\mathcal{E}_1,...,\mathcal{E}_n\}$ consists of exceptional objects it is called \emph{exceptional collection}.}
\end{defi}
\begin{exam}
	\textnormal{Let $\mathbb{P}^n$ be the projective space and consider the ordered collection of invertible sheaves $\{\mathcal{O}_{\mathbb{P}^n}, \mathcal{O}_{\mathbb{P}^n}(1),...,\mathcal{O}_{\mathbb{P}^n}(n)\}$. In \cite{BES} Beilinson showed that this is a full strong exceptional collection. }
\end{exam}
A generalization of the notion of a full w-exceptional collection is that of a semiorthogonal decomposition of $D^b(X)$. Recall that a full triangulated subcategory $\mathcal{D}$ of $D^b(X)$ is called \emph{admissible} if the inclusion $\mathcal{D}\hookrightarrow D^b(X)$ has a left and right adjoint functor. 
\begin{defi}
	\textnormal{Let $X$ be a smooth projective variety over $k$. A sequence $\mathcal{D}_1,...,\mathcal{D}_n$ of full triangulated subcategories of $D^b(X)$ is called \emph{semiorthogonal} if all $\mathcal{D}_i\subset D^b(X)$ are admissible and $\mathcal{D}_j\subset \mathcal{D}_i^{\perp}=\{\mathcal{F}\in D^b(X)\mid \mathrm{Hom}(\mathcal{G},\mathcal{F})=0$, $\forall$ $ \mathcal{G}\in\mathcal{D}_i\}$ for $i>j$. Such a sequence defines a \emph{semiorthogonal decomposition} of $D^b(X)$ if the smallest full thick subcategory containing all $\mathcal{D}_i$ equals $D^b(X)$.}
\end{defi}
\noindent
For a semiorthogonal decomposition we write $D^b(X)=\langle \mathcal{D}_1,...,\mathcal{D}_n\rangle$.
\begin{exam}
	\textnormal{Let $\mathcal{E}_1,...,\mathcal{E}_n$ be a full w-exceptional collection on $X$. It is easy to verify that by setting $\mathcal{D}_i=\langle\mathcal{E}_i\rangle$ one gets a semiorthogonal decomposition $D^b(X)=\langle \mathcal{D}_1,...,\mathcal{D}_n\rangle$.}
\end{exam}

The noncommutative motives $M_{T(X)}$ and $M_{T(Y)}$ are defined in the introduction.

\begin{thm2}[Theorem 1.7]
Let $X$ and $Y$ be schemes of pure weak exceptional type. 
If $X$ and $Y$ are birational, then there are direct summands $N,N'\in \mathrm{CSA}^{\oplus }$ of $M_{T(X)}$ and $Q,Q'\in \mathrm{CSA}^{\oplus }$ of $M_{T(Y)}$ such that $U(X)\oplus N'\oplus Q\simeq U(Y)\oplus Q'\oplus N$.
\end{thm2}
\begin{proof}
	The semiorthogonal decompositions 
	\begin{eqnarray*}
		D^b(X)=\langle \mathcal{E}_1,...,\mathcal{E}_m\rangle
	\end{eqnarray*}	
	and
	\begin{eqnarray*}
		D^b(Y)=\langle \mathcal{F}_1,...,\mathcal{F}_n\rangle
	\end{eqnarray*}
	imply $\mathrm{Pic}(X)\simeq \mathbb{Z}^{\oplus m}$ and $\mathrm{Pic}(Y)\simeq \mathbb{Z}^{\oplus n}$. Since $\mathcal{E}_i$ is pure of type $\mathcal{K}_i$, we conclude that $\langle\mathcal{K}_1,...,\mathcal{K}_m\rangle$ is a semiorthogonal decomposition of $D^b(X_s)$ which is induced from a full exceptional collection. Therefore $\mathrm{Pic}(X_s)\simeq \mathbb{Z}^{\oplus m}$. In the same way it follows $\mathrm{Pic}(Y_s)\simeq \mathbb{Z}^{\oplus n}$. Some of the $\mathcal{K}_i$ form a basis of $\mathrm{Pic}(X_s)$. Without loss of generality, let $\mathcal{K}_1,...,\mathcal{K}_r$ be a basis. Then, by assumption, there are pure vector bundles $\mathcal{E}_1,...,\mathcal{E}_r$ of type $\mathcal{K}_1,...,\mathcal{K}_r$. Therefore, the assumptions of Theorem 4.5 are fulfilled. Let $D_X$ (resp. $D_Y$) denote the set of indecopmposable $AS$-bundles on $X$ (resp. $Y$). Since $X$ satisfies the assumptions of Theorem 4.5, the proof of Theorem 1.1 shows that for any indecomposable $AS$-bundle $\mathcal{E}$ on $X$ one has $[\mathrm{End}(\mathcal{E})]\in \mathrm{Br}(k)$. From Theorem 4.5 and Proposition 4.7 we conclude that there are only finitely many Brauer-classes $[\mathrm{End}(\mathcal{E})]\in \mathrm{Br}(k)$ of indecomposable $AS$-bundles. 
	Now let $C_X\subset \mathrm{Br}(k)$ be the subgroup generated by these finitely many Brauer-classes. Analogously, we define $C_Y\subset\mathrm{Br}(k)$. Denote by $A(g)$ the central simple division algebra corresponding to $g\in C_X$ and by $A(h)$ the central simple division algebra corresponding to $h\in C_Y$.  As in the introduction, we set 
	\begin{eqnarray*}
		M_{T(X)}:=\bigoplus _{g\in C_X}U(A(g))\quad \textnormal{and} \quad M_{T(Y)}:=\bigoplus _{h\in C_Y}U(A(h)).
	\end{eqnarray*}
Furthermore, we put 
	\begin{eqnarray*}
	M_{X}:=\bigoplus _{p\in \mathrm{Am}(X)}U(A_{p})\quad \textnormal{and} \quad M_{Y}:=\bigoplus _{q\in \mathrm{Am}(Y)}U(A_{q})
\end{eqnarray*}
where $A_p$ is the central simple division algebra corresponding to $p$ and $A_q$ the central simple division algebra corresponding to $q$.
Consider the simiorthogonal decompositions
\begin{eqnarray*}
	 D^b(X)=\langle \mathcal{E}_1,...,\mathcal{E}_m\rangle
	\end{eqnarray*}	
and 
\begin{eqnarray*}
	D^b(Y)=\langle \mathcal{F}_1,...,\mathcal{F}_n\rangle.
\end{eqnarray*}
By assumption, the vector bundles $\mathcal{E}_1,...,\mathcal{E}_m$ and $\mathcal{F}_1,...,\mathcal{F}_n$ are pure having endomorphism algebras being isomorphic to $k^s$. This implies that $\mathrm{End}(\mathcal{E}_i)$ and $\mathrm{End}(\mathcal{F}_j)$ are central simple algebras (see proof of \cite{NO}, Proposition 3.3). From the construction of noncommutative motives we have
	\begin{eqnarray*}
	U(X):=\bigoplus^m _{i=1}U(\mathrm{End}(\mathcal{E}_i))\quad \textnormal{and} \quad U(Y):=\bigoplus^n_{j=1}U(\mathrm{End}(\mathcal{F}_j)).
\end{eqnarray*}
Obviously, $U(X)$ is a direct summand of $M_{T(X)}$ and therefore there exists $N'\in\mathrm{CSA}(k)^{\oplus}$ such that $U(X)\oplus N'\simeq M_{T(X)}$. In the same way one shows that there is a $Q'\in\mathrm{CSA}(k)^{\oplus}$ such that $U(Y)\oplus Q'\simeq M_{T(Y)}$. Note that $\mathrm{Am}(X)$ is a subgroup of $C_X$. Hence there exists a $N$ such that $M_{T(X)}\simeq M_X\oplus N$. The same holds for $M_Y$. This gives us $U(X)\oplus N'=M_X\oplus N$ and $U(Y)\oplus Q'=M_Y\oplus Q$, implying the equalities $M_X\oplus N\oplus Q=U(X)\oplus N'\oplus Q$ and $M_Y\oplus Q\oplus N=U(Y)\oplus Q'\oplus N$. 
Now if $X$ and $Y$ are birational, it follows $\mathrm{Am}(X)=\mathrm{Am}(Y)$ and therefore $M_X\simeq M_Y$. Hence $U(X)\oplus N'\oplus Q\simeq U(Y)\oplus Q'\oplus N$.
This completes the proof.	
\end{proof}
\begin{cor2}[Corollary 1.8]
	Let $X$ and $Y$ be as in Theorem 1.7. Assume $X$ and $Y$ are birational and let $A_i, 1\leq i\leq n$ and $B_j, 1\leq j\leq m$ be the central simple algebras occuring in $U(X)\oplus N'\oplus Q$ and $U(Y)\oplus Q'\oplus N$ respectively, then $\langle [A_i]\rangle=\langle[B_j]\rangle$ in $\mathrm{Br}(k)$.
\end{cor2}
\begin{proof}
	This follows from \cite{TA2S}, Corollary 4.8.
\end{proof}

\vspace{0.5cm}
\noindent
{\tiny HOCHSCHULE FRESENIUS UNIVERSITY OF APPLIED SCIENCES 40476 D\"USSELDORF, GERMANY.}
E-mail adress: sasa.novakovic@hs-fresenius.de\\
\noindent
{\tiny MATHEMATISCHES INSTITUT, HEINRICH--HEINE--UNIVERSIT\"AT 40225 D\"USSELDORF, GERMANY.}
E-mail adress: novakovic@math.uni-duesseldorf.de

\end{document}